\documentclass[11pt, oneside]{amsart}

\usepackage[english]{babel} 
\usepackage{graphics,color,pgf}
\usepackage{epsfig}
\usepackage[ansinew]{inputenc}
\usepackage[all]{xy}
\usepackage{hyperref}
\usepackage{tikz-cd}
\usetikzlibrary{matrix,arrows,decorations.pathmorphing}
\newdir{ >}{!/8pt/@{}*@{>}}
\usepackage{amssymb, amsmath,amsthm, mathtools, amscd}
\usepackage{mathrsfs}
\usepackage[margin=1.5in]{geometry}

\makeatletter
\newtheorem*{rep@theorem}{\rep@title}
\newcommand{\newreptheorem}[2]{%
\newenvironment{rep#1}[1]{%
 \def\rep@title{#2 \ref{##1}}%
 \begin{rep@theorem}}%
 {\end{rep@theorem}}}
\makeatother

\newtheorem{intro_thm}{Theorem}

\newtheorem{intro_prop}[intro_thm]{Proposition}

\newtheorem{lemma}{Lemma}[section]
 
\newtheorem{prop}[lemma]{Proposition}
\newtheorem{cor}[lemma]{Corollary}

\newtheorem{setup}[lemma]{Setup}

\theoremstyle{definition}
\newtheorem{defn}[lemma]{Definition}

\theoremstyle{remark}
\newtheorem{oss}[lemma]{Remark}

\newcommand\matH{{\mathbb{H}}^p_{\mathbb{C}}}

\newcommand\matR{{\mathbb{R}}}

\newcommand\matC{{\mathbb{C}}}

\newcommand\calS{{\mathcal S}_{m,n}}

\newcommand\calQ{{\mathcal Q}}
\newcommand\calC{{\mathcal C}}
\newcommand\calB{{\mathcal B}}

\newcommand{\pu}{\textup{PU}(p,1)}
\newcommand{\su}{\textup{SU}(m,n)}

\newcommand{\Hm}{\textup{H}}
\newcommand{\Hb}{\textup{H}_{\text{b}}}
\newcommand{\Hc}{\textup{H}_{\text{c}}}
\newcommand{\Hcb}{\textup{H}_{\text{cb}}}

\newcommand{\Linf}{\text{L}^{\infty}}

\newcommand{\Tbbullet}{\textup{T}_b^{\bullet}}

\newcommand{\Tbdue}{\textup{T}_b^2}

\newcommand{\rk}{\text{rk}}
\newcommand{\Meas}{\text{Meas}}

\newcommand{\Isom}{\text{Isom}}

\newcommand{\Stab}{\text{Stab}}

\newcommand{\Prob}{\text{Prob }}

\begin{document}

\title[Superrigidity maximal measurable cocycles]{Superrigidity of maximal measurable cocycles of complex hyperbolic lattices}

\author[F. Sarti]{F. Sarti}
\address{Department of Mathematics, University of Bologna, Piazza di Porta San Donato 5, 40126 Bologna, Italy}
\email{filippo.sarti8@unibo.it}

\author[A. Savini]{A. Savini}
\address{Section de Math\'ematiques, University of Geneva, Rue Du Li\`evre 2, Geneva 1227, Switzerland}
\email{alessio.savini@unige.ch}

\date{\today.\ \copyright{\ F. Sarti, A. Savini 2021}. The first author 
was supported through GNSAGA, funded by the INdAM. The second  author 
was partially supported by the project Geometric and harmonic analysis with applications, funded by EU Horizon 2020 under the Marie Curie Grant Agreement No. 777822. 
}

\begin{abstract}
Let $\Gamma$ be a torsion-free lattice of $\pu$ with $p \geq 2$ and let $(X,\mu_X)$ be an ergodic standard Borel probability $\Gamma$-space. 
We prove that any maximal Zariski dense measurable cocycle $\sigma: \Gamma \times X \longrightarrow \su$
is cohomologous to a cocycle associated to a representation of $\pu$ into $\su$, with $1 \leq m \leq n$.
The proof follows the line of Zimmer' Superrigidity Theorem and requires the existence of a boundary map, that we prove in a much more general setting. As a consequence of our result, there cannot exist maximal measurable cocycles with the above properties when $1 < m < n$.

\end{abstract}
  
\maketitle

\section{Introduction}


Given $\Gamma \leq L$ a torsion-free lattice of a semisimple Lie group, a fruitful way to study representation spaces of $\Gamma$ is based on numerical invariants coming from bounded cohomology. Typically those invariants have bounded absolute value and representations attaining the maximum (that is \emph{maximal representations}) are all conjugated to some fixed representation of the ambient group. For instance, in the particular case of surface groups, Goldman \cite{Goldth} studied the relation between the Teichm\"uller space and the maximality of the Euler invariant. Indeed the maximality of the latter corresponds to the choice of a specific component of the $\textup{PSL}(2,\matR)$-character variety.  By substituting $\textup{PSL}(2,\matR)$ with any Hermitian Lie group $G$, Burger, Iozzi and Wienhard \cite{BIW09} gave a structure theorem for \emph{tight} representations of locally compact groups into $G$. A representation is called tight if the map induced in bounded cohomology preserves the norm of the K\"ahler class $\kappa^b_G$ associated to $G$. Later, the same authors focused their attention on surface groups (also with boundary components) and studied systematically maximal representations into Hermitian Lie groups leading to a complete characterization of such representations \cite{BIW1}. 

In this paper we are going to focus our attention on complex hyperbolic lattices, that is torsion-free lattices of $\pu$, with $p \geq 2$. The theory of maximal representations of complex hyperbolic lattices has been widely studied so far. For instance, when $\Gamma$ is a \emph{non-uniform} lattice, Koziarz and Maubon \cite{koziarz:maubon} showed that maximal representations into the group $\textup{PU}(q,1)$, where $q \geq p \geq 2$, must be conjugated to the standard lattice embedding $\Gamma \rightarrow \pu$, possibly modulo a compact subgroup when $q>p$. This result was obtained by using techniques relying on the theory of harmonic maps, whereas Burger and Iozzi \cite{BIcartan} proved the same result using a bounded cohomology approach. 

Also some superrigidity phenomena are given. For instance, Pozzetti \cite{Pozzetti} proved that maximal representations into $\su$ with Zariski dense image must be induced by representations of the ambient group $\pu$. In particular, this suggests that when $1 < m < n$ such representations cannot exist. Additionally, she gave also a structure theorem for maximal representations, following the line of some previous works by Hamlet and herself \cite{Ham13,Ham,HP}. More recently Koziarz and Maubon \cite{KM17} refined such a characterization with the use of Higgs bundles. 

Inspired by the work by Bader, Furman and Sauer \cite{sauer:articolo} for couplings and by Burger and Iozzi \cite{burger:articolo} for representations, one of the author together with Moraschini \cite{savini3:articolo,moraschini:savini,moraschini:savini:2,savini2020} has recently developed the theory of numerical invariants of measurable cocycles, that is part of a program whose aim is to generalize rigidity results for representations of rank-one lattices. Those invariants are obtained by pulling back preferred classes in bounded cohomology along cocycles, imitating the case of representations. The aim of this paper is to apply this machinery to the study of measurable cocycles of complex hyperbolic lattices to get a superrigidity result similar to the one obtained by Pozzetti. 
More precisely, let $\Gamma < \pu$ be a torsion-free lattice, $(X,\mu_X)$ be an ergodic standard Borel probability $\Gamma$-space and $\sigma:\Gamma\times X\rightarrow \su$ be  a measurable cocycle. 
Exploiting the K\"{a}hler form on the symmetric space associated to $\su$, one can define the \emph{bounded K\"{a}hler class} of $\su$ in $\Hcb^2(\su;\matR)$, that we denote by $k_{\su}^b$. The pullback of such a class along $\sigma$ gets back a class in $\Hb^2(\Gamma;\matR)$ and through the transfer map we obtain a class in $\Hcb^2(\pu;\matR)$.  Hence we can consider the multiplicative constant obtained by comparing such a class with the \emph{Cartan class} in $\Hcb^2(\pu;\matR)$. This turns out to be a numerical invariant $\textup{t}_b(\sigma)$ that we are going to call \emph{Toledo invariant} associated to $\sigma$. Since the latter has absolute value bounded by the rank of $\su$, it makes sense to define the notion of \emph{maximal measurable cocycles} as those ones with maximal Toledo invariant.

If in addition $\sigma$ admits a \emph{boundary map} into the Shilov boundary $\calS$ of $\su$, namely a measurable map $\phi:\partial\matH \times X \rightarrow \calS$ which is $\sigma$-equivariant,  the machinery of \cite{moraschini:savini}, \cite{moraschini:savini:2} can be exploited to get a useful representative for the pull back of $k^b_{\su}$.
More precisely, we recall that $\calS$ can be identified with the quotient of $\su$ by a maximal parabolic subgroup stabilizing a maximal isotropic subspace of $\mathbb{C}^{n+m}$. 
On this boundary one can define naturally a bounded measurable $\su$-invariant cocycle, the \emph{Bergmann cocycle}, that corresponds to the bounded K\"{a}hler class through the canonical map 
$\Hm^2(\calB^{\infty}((\calS)^{\bullet+1};\matR))\rightarrow \Hcb^2(\su;\matR)$ defined by Burger and Iozzi \cite{burger:articolo}.
Hence we can apply the theory of pullback along boundary maps getting a class in $\Hm^2(\Linf((\partial_{\infty}\matH)^{\bullet+1};\matR)^{\Gamma})\cong \Hb(\Gamma;\matR)$ that coincides with the pullback of $k_{\su}^b$. 

Even though the definition of Toledo invariant for measurable cocycles $\Gamma \times X \rightarrow \su$ can be given without boundary maps, those maps are fundamental in order to adapt the proof by Zimmer. For this reason our first result, which is given in a much more general setting, investigates the existence of boundary maps.
In the following statement a $\Gamma$-boundary is an amenable $\Gamma$-space which is relatively metrically ergodic (see Definition \ref{def:rel:iso:ergodic})

\begin{intro_thm}\label{thm_boundary_map}
Let $\Gamma$ be a locally compact and second countable group and let $H$ be a simple Lie group of non-compact type. Let $(X,\mu_X)$ be an ergodic standard Borel probability $\Gamma$-space and let $\sigma:\Gamma\times X \rightarrow H$ be a Zariski dense measurable cocycle. Then, for any $\Gamma$-boundary $B$ there exists a $\sigma$-equivariant map $\phi:B\times X\rightarrow H/P$ where $P<H$ is a minimal parabolic subgroup.
\end{intro_thm}

The proof of the above result relies on the techniques by Bader and Furman (\cite{BFalgebraic}) concerning the category of algebraic representations of ergodic spaces. A fundamental tool will be the notion of relative metric ergodicity, that we recall in Section \ref{section_relative_isometric_ergodic}. Here the assumption of Zariski density of $\sigma$ refers to the fact that the algebraic hull coincides with the whole target group. 

As a consequence of Theorem \ref{thm_boundary_map}, given a torsion-free lattice $\Gamma < \pu$ and an ergodic standard Borel probability $\Gamma$-space $(X,\mu_X)$, any Zariski dense measurable cocycle $\sigma:\Gamma \times X \rightarrow \su$ admits a boundary map $\phi:\partial_\infty \mathbb{H}^p_{\mathbb{C}} \times X \rightarrow \calS$. We are going to prove even more. Indeed, exploiting the ergodicity of the space $X$, we will show that the slice $\phi_x(\xi):=\phi(\xi,x)$ is \emph{essentially Zariski dense} for almost every $x \in X$ (see Proposition \ref{prop:zariski:density}). 

Finally notice that Theorem \ref{thm_boundary_map} has important consequences also on the work of one of the author \cite{savini2020}, answering to the question of \cite[Section 4.1]{savini2020} about the existence of boundary maps for Zariski dense measurable cocycles of surface groups.

Thanks to the existence of a boundary map, we are going to prove the following superrigidity result.

\begin{intro_thm}\label{main_theorem}
Consider $p \geq 2$ and $1 \leq m \leq n$. Let $\Gamma<\pu$ be a torsion-free lattice and let $(X,\mu_X)$ be an ergodic standard Borel probability $\Gamma$-space. If $\sigma: \Gamma \times X\rightarrow \su$ is a maximal Zariski dense measurable cocycle, then it is cohomologous to the restriction of a cocycle associated to a representation $\rho: \pu \rightarrow \su$.
\end{intro_thm}

In the particular case when $m=1$, we get back \cite[Theorem 3]{moraschini:savini:2} for Zariski dense cocycles. We also point out that, under the Zariski density assumption, Theorem \ref{main_theorem} does not follow by Zimmer' Superrigidity Theorem \cite[Theorem 4.1]{zimmer:annals} since the rank of $\pu$ is one. So our result should be considered as a suitable adaptation of Zimmer's theorem in the context of rank-one lattices, where to get the desired trivialization one needs to introduce the maximality assumption. 

It is worth noticing that it was conjectured that maximal representations of a complex hyperbolic lattice into a Hermitian Lie group are superrigid. Theorem \ref{main_theorem} suggests that a similar conjecture should hold also for measurable cocycles.


As an easy application of the results by Pozzetti, we also obtain the following

\begin{intro_prop}\label{corollary}
Consider $p \geq 2$. Let $\Gamma<\pu$ be a torsion-free lattice and let $(X,\mu_X)$ be an ergodic standard Borel probability $\Gamma$-space. Assuming $1<m<n$, there is no maximal Zariski dense measurable cocycle $\sigma:\Gamma \times X\rightarrow\su$. 
\end{intro_prop}

The proof of Theorem \ref{main_theorem} is not a mere adaptation of Pozzetti's proof. We need an additional study of the slices of the boundary map. Indeed, since the slice $\phi_x$ is measurable for almost every $x \in X$ \cite[Chapter VII, Lemma 1.3]{margulis:libro}, the maximality of $\sigma$ implies that the slice $\phi_x$ preserves the chain geometry for almost every $x \in X$. Being essentially Zariski dense by Proposition \ref{prop:zariski:density}, by \cite[Theorem 1.6]{Pozzetti} we argue that $\phi_x$ is a rational map for almost every $x \in X$. This gives us back a measurable map $\Phi:X \rightarrow \textup{Rat}(\partial\matH,\calS) , \  \Phi(x):=\phi_x$. Thus following the line of the proof of \cite[Theorem 4.1]{zimmer:annals}, we exploit the ergodicity of $\Gamma$ on $X$ and the smoothness of the action of $\pu \times \su$ on $\textup{Rat}(\partial\matH,\calS)$ to get the desired statement. In the proof we will exploit the fact that stabilizers of measures on quotients of $\pu$ by a closed subgroup are algebraic (see Lemma \ref{lemma_measure}). 

\vspace{ 0.5 cm}
\paragraph{\textbf{Plan of the paper.}} The paper is divided into four sections. In Section \ref{section_preliminaries} we recall some preliminary definitions and known results. More precisely, we start by giving the definition of measurable cocycle, we remark how it extends the notion of representation and we generalize conjugation by defining what cohomologous cocycles are. Afterwards, we focus on bounded cohomology of locally compact groups, recalling basic definitions and Burger-Monod's functorial approach which provides a useful technique for the computation. Then we move to the definition of two specific cohomology classes, the \emph{K\"{a}hler class} and the \emph{Cartan class}. Finally we remind the notion of smooth, ergodic and amenable actions. We conclude with the notions of relative metric ergodicity and boundaries. 

The aim of Section \ref{section_toledo} is to introduce the \emph{Toledo invariant} of a measurable cocycle and to prove some of its properties. Using boundary maps to implement the pullback in bounded cohomology and applying the transfer map, we are allowed to compare the pullback of the K\"{a}hler class with the Cartan class. The real number obtained by such a  comparison will be our desired invariant. Then we prove that the module of such invariant is bounded by the rank $\rk (\mathcal{X}_{m,n})$ of the symmetric space associated to $\su$.

In Section \ref{section_boundary_map} we prove Theorem \ref{thm_boundary_map} and Proposition \ref{prop:zariski:density}, where we investigate the relation between the Zariski density of the cocycle and the essential Zariski density of the slices of the boundary map.

Section \ref{section_proof} is spent to prove the Theorem \ref{main_theorem} and Proposition \ref{corollary}.

\vspace{ 0.5 cm}
\paragraph{\textbf{Acknowledgements.}} We would like to thank Maria Beatrice Pozzetti, Jonas Beyer, Alessandra Iozzi, Michelle Bucher and Stefano Francaviglia for the enlightening conversations and the useful comments about our work. 

We want to thank also Bruno Duchesne for the interest he showed in this result. 

We are finally grateful to the anonymous referee for her/his suggestions which allowed us to improve the quality of our paper.

\section{Preliminaries}\label{section_preliminaries}

\subsection{Measurable cocycles}\label{section_cocycles}

Let $G$ and $H$ be locally compact groups endowed with their Haar measurable structures. Let
$(X,\mu_X)$ be a standard Borel probability space equipped with a measure preserving $G$-action. Additionally, suppose that $\mu_X$ is atom-free.
Under those assumptions, we say that $(X,\mu_X)$ is a \emph{standard Borel probability $G$-space}.  If the $G$-action preserves only the measure class of $\mu_X$, we are going to call $(X,\mu_X)$ a \emph{Lebesgue $G$-space}. 

Given another measure space $(Y,\mu_Y)$, we denote by $\Meas (X,Y)$ the \emph{space of measurable maps} from $X$ to $Y$, where we identify two measurable functions if they coincide 
up to a measure zero subset. We endow $\Meas(X,Y)$ with the topology of convergence in measure. In the previous setting we can give the following

\begin{defn}\label{cocycle_definition}
 A \emph{measurable cocycle} is a measurable function $\sigma: G\times X\rightarrow H$ which satisfies the cocycle condition 
 \begin{equation}\label{cochain_condition}
 \sigma(g_1g_2,x)=\sigma(g_1,g_2x)\sigma(g_2,x)
 \end{equation}
 holds for almost every $g_1,g_2\in G$ and for almost every $x\in X$. 
\end{defn}

The cocycle condition \eqref{cochain_condition} can be suitably interpreted as a generalization of the chain rule for differentiation. Moreover, it is equivalent to the equation defining Borel 1-cocycles of $\textup{Meas}(G,\textup{Meas}(X,H))$ in the sense of Eilenberg-MacLane (see for instance \cite{feldman:moore}).

The notion of measurable cocycle is quite ubiquitous in mathematics, but for our purposes we will mainly focus our attention on how measurable cocycles extend the concept of representations. More precisely, given a representation one can naturally define a cocycle associated to it as follows.

\begin{defn}\label{def_induced_cocycle}
 Let $\rho:G\rightarrow H$ be a continuous representation and let $(X,\mu_X)$ be a standard Borel probability
 $G$-space. The \emph{cocycle associated to} $\rho$ is defined by
 $$\sigma_{\rho}(g,x)\coloneqq\rho(g)$$
 for every $g\in G$ and for almost every $x\in X$.
\end{defn}

Notice that when $\Gamma$ is a lattice, it naturally inherits the 
discrete topology from the ambient group and hence any representation is automatically continuous. 

By following the interpretation of measurable cocycles as Borel $1$-cocycles, we 
now introduce the notion of cohomologous cocycles. 

\begin{defn}\label{def_cohomologous_cocycle}
 Let $\sigma_1,\sigma_2:G\times X \rightarrow H$ be two measurable cocycles, let 
 $f:X\rightarrow H$ be a measurable map and denote by $\sigma_1^f$ the cocycle defined as
 $$\sigma_1^f(g,x)\coloneqq f(gx)^{-1}\sigma_1(g,x)f(x)$$
 for every $g\in G$ and almost every $x\in X$. The cocycle $\sigma^f_1$ is the \emph{$f$-twisted cocycle associated to $\sigma_1$}. 
 We say that $\sigma_1$ is \emph{cohomologous} to $\sigma_2$ if there exists a measurable map $f$
 such that $\sigma_2=\sigma_1^f$.
\end{defn}
As well measurable cocycles may be interpreted as a generalization of representations, so the notion of cohomologous cocycles actually extends conjugacy between representations.

An important tool that may encode useful information about a given representation is the closure of its image. One would like to introduce a similar notion in the context of measurable cocycles. Unfortunately this attempt may reveal quite difficult since a priori the image of a cocycle does not have any nice algebraic property like a group structure. Nevertheless, when the target is an algebraic group, it is possible to introduce the notion of algebraic hull.

\begin{defn}
Let \textbf{H} be a semisimple real algebraic group and denote by $H=\mathbf{H}(\mathbb{R})$ the real points of $\mathbf{H}$. The \emph{algebraic hull} of a measurable cocycle $\sigma: G \times X \rightarrow H$ is the (conjugacy class of
the) smallest algebraic subgroup $\textbf{L}$ of $\textbf{H}$ such that $\textbf{L}(\matR)^{\circ}$ contains the image of a
cocycle cohomologous to $\sigma$.

We say that $\sigma$ is \emph{Zariski dense} if its algebraic hull coincides with the whole group $\mathbf{H}$. 
\end{defn}

The fact that the notion of algebraic hull is well-defined follows by the descending chain condition in $\textbf{H}$, which is algebraic and hence Noetherian \cite[Proposition 9.2.1]{zimmer:libro}. Notice that by \cite[Proposition 3.1.6]{zimmer:libro} simple Lie groups are real points of algebraic groups, thus it makes sense to speak about Zariski density in that case. This should clarify the assumption of Theorem \ref{thm_boundary_map}

%
%

\subsection{Continuous bounded cohomology}\label{section_cohomology}

In this section we are going to recall both the definitions of continuous and continuous bounded cohomology for locally compact groups. We refer to \cite{monod:libro,burger2:articolo} for more details. 

We define the sets of \emph{real continuous functions on} $G^{\bullet+1}$ as
$$\calC_c^{\bullet}(G;\mathbb{R})\coloneqq \{f:G^{\bullet+1}\rightarrow \mathbb{R} \ | \  f \text{ continuous}\}$$
and the maps
$$\delta^{\bullet}:\calC_c^{\bullet}(G;\mathbb{R})\rightarrow\calC_c^{\bullet+1}(G;\mathbb{R})$$
where
$$\delta^{\bullet}(f)(g_0,\cdots,g_{\bullet+1})\coloneqq\sum\limits_{i=0}^{\bullet+1} (-1)^if(g_0,\cdots,g_{i-1},g_{i+1},\cdots,g_{\bullet+1}) \ .$$
There exists a natural $G$-action on $\calC_c^{\bullet}(G;\mathbb{R})$ defined by
\begin{equation}\label{eq_action}
(gf)(g_0,\cdots,g_{\bullet})\coloneqq f(g^{-1}g_0,\cdots,g^{-1}g_{\bullet}) 
\end{equation}
for every $f\in\calC_c^{\bullet}(G;\mathbb{R})$ and for every $g,g_0,\ldots,g_{\bullet}\in G$.
Hence, if we consider the sets of $G$-invariant cochains
$$\calC_c^{\bullet}(G;\mathbb{R})^G\coloneqq \{f\in \calC_c^{\bullet}(G;\mathbb{R}) \ | \  gf=f\} \ ,$$
the restriction $\delta^{\bullet}_|$ of the coboundary operator $\delta^{\bullet}$ is well-defined, since it preserves $G$-invariance. The pair 
$$(\calC_c^{\bullet}(G;\mathbb{R})^G,\delta^{\bullet}_|)$$
is the \emph{cochain complex of real-valued invariant continuous functions on $G$}.
\begin{defn}
 The \emph{continuous cohomology} of the group $G$ with real coefficients is the 
 cohomology of the complex $(\calC_c^{\bullet}(G;\mathbb{R})^G,\delta^{\bullet}_|)$ and it is denoted by $\Hc^{\bullet}(G;\mathbb{R})$.
\end{defn}
Similarly, we can consider the subspace of continuous bounded cochains, that is the cochain complex given by
$$(\calC_{cb}^{\bullet}(G;\mathbb{R})^G,\delta^{\bullet}_|)$$
where
$$\calC_{cb}^{\bullet}(G;\mathbb{R})^G\coloneqq\left\{f\in \calC_c^{\bullet}(G;\mathbb{R})^G \ | \ \sup_{g_0,\cdots,g_{\bullet}} |f(g_0,\cdots,g_{\bullet})|<+\infty\right\}$$
and the coboundary is obtained by restriction, since it preserves boundedness.

\begin{defn}
The \emph{continuous bounded cohomology} of the group $G$ with real coefficients is the 
 cohomology of the complex $(\calC_{cb}^{\bullet}(G;\mathbb{R})^G,\delta^{\bullet}_|)$ and it is denoted by $\Hcb^{\bullet}(G;\mathbb{R})$. 
\end{defn}


It is worth mentioning that the space $\Hcb^\bullet(G;\mathbb{R})$ admits a standard seminormed structure. Indeed, given an element $\alpha\in \Hcb^{\bullet}(G;\mathbb{R})$,
we can define its seminorm as follows
$$\lVert \alpha \rVert_\infty \coloneqq \inf \{||c||_{\infty} \ | \ [c]=\alpha \}\ . $$
In this way $\Hcb^{\bullet}(G;\mathbb{R})$ becomes a seminormed space with the quotient seminorm.

The computation of continuous bounded cohomology groups may reveal quite complicated using the standard definition. Burger and Monod \cite{monod:libro,burger2:articolo} circumvent this problem showing a way to compute continuous bounded cohomology with the use of strong resolutions by relatively injective $G$-modules. Since it would be too technical to introduce those notions here, we prefer to omit them and we refer to Monod's book \cite{monod:libro} for a more detailed exposition. The result by Burger and Monod \cite[Corollary 1.5.3.]{burger2:articolo} shows that given a locally compact group $G$ and a strong resolution of $\mathbb{R}$
 by relatively injective Banach $G$-modules $(E^{\bullet},\delta^{\bullet})$, there exists an isomorphism 
$$\Hcb^k (G;\mathbb{R})\cong \textup{H}^k((E^{\bullet})^G) \ ,$$
for every $k \in \mathbb{N}$. Here $(E^\bullet)^G$ is the subcomplex of $G$-invariant vectors in $E^\bullet$. Unfortunately the isomorphism above is not isometric a priori, in the sense that it does not necessarily preserve seminorms.

We will spend the rest of the section to define a strong resolution which actually realizes the isomorphism isometrically. Let $(S,\mu)$ be an amenable $G$-space (we refer the reader to Definition \ref{amenable_action}). The cochain complex of \emph{essentially bounded measurable functions on $S^{\bullet+1}$} is $(\Linf (S^{\bullet+1};\mathbb{R}),\delta^{\bullet})$, where $\delta^\bullet$ is the standard homogeneous coboundary operator. If we consider the $G$-action defined by Equation \eqref{eq_action} and we complete the previous complex with the inclusion of coefficients $\mathbb{R} \hookrightarrow \Linf(S;\mathbb{R})$, we obtain a resolution of $\mathbb{R}$ that is strong and consists of relatively injective $G$-modules \cite[Theorem 1]{burger2:articolo}. Hence, by \cite[Corollary 1.5.3]{burger2:articolo}, we have the following isomorphism
$$\Hcb^k (G;\mathbb{R})\cong \textup{H}^k (\Linf(S^{\bullet+1};\mathbb{R})^G) \ ,$$
for every $k \in \mathbb{N}$. The striking result is that this isomorphism is in fact isometric \cite[Theorem 2]{burger2:articolo}.

Even if $S$ is not amenable, it is sufficient to consider the complex $(\mathcal{B}^\infty(S^{\bullet+1};\mathbb{R}),\delta^\bullet)$ of \emph{bounded measurable functions on $S^{\bullet+1}$} to gain information about the continuous bounded cohomology of $G$. Indeed with the $G$-action defined by Equation \eqref{eq_action} and the inclusion $\mathbb{R} \hookrightarrow \mathcal{B}^\infty(S;\mathbb{R})$, we obtain a resolution of $\mathbb{R}$ which is strong (but not necessarily by relatively injective modules). By Burger and Iozzi \cite[Corollary 2.2]{burger:articolo} there exists a canonical map 
$$
\mathfrak{c}^k:\textup{H}^k(\mathcal{B}^\infty(S^{\bullet+1};\mathbb{R})^G) \rightarrow \Hcb^k(G;\mathbb{R}) 
$$
for every $k \in \mathbb{N}$. We will tacitly exploit the previous result to ensure that the pullback of measurable cochain along boundary maps lies in $\Linf$.

\subsection{Cartan and K\"{a}hler classes}\label{section_bergmann_cartan}

In this section we are going to define two bounded measurable cocycles that will be the main ingredients
for the definition of the Toledo invariant associated to a measurable cocycle. 

Let $1 \leq m \leq n$. Let $\su$ be the subgroup of $\textup{SL}(n+m,\matC)$ which preserves the Hermitian form defined by the following matrix 
\begin{equation*}
 h=\begin{bmatrix}
  \textup{Id}_m & \\
  &-\textup{Id}_{n}
 \end{bmatrix} \ ,
\end{equation*}
where the matrices appearing above are the identity matrix of order given by the subscript. Denote by $\mathcal{X}_{m,n}$ the associated symmetric space. It is well-known that the latter is a \emph{Hermitian} symmetric space, that is it admits a $\su$-invariant complex structure. Additionally, when $m=n$, the Hermitian space is \emph{of tube type}, namely it can be biholomorphically realized as a domain of the form $V+i \Omega$, where $V$ is a real vector space and $\Omega \subseteq V$ is a proper convex cone. More generally $\mathcal{X}_{m,n}$ contains maximal tube type subdomains which are all isometric to the symmetric space $\mathcal{X}_{m,m}$ assuming $1 \leq m \leq n$. 

The space $\mathcal{X}_{m,n}$ can also be realized biholomorphically as a bounded convex subspace of $\matC^n$. In that case, $\su$ acts via biholomorphisms and this action extends continuously to the boundary $\partial \mathcal{X}_{m,n}$, which is not a homogeneous $\su$-space but it admits a unique closed $\su$-orbit called \emph{Shilov boundary}. The Shilov boundary $\calS$ is the smallest closed subset on which one can apply the maximum principle and it can be identified with the quotient $\su/Q$, where $Q$ is a \emph{maximal} parabolic subgroup stabilizing a maximal isotropic subspace of $\mathbb{C}^{m+n}$. Notice that when $1 \leq m \leq n$, the Shilov boundaries of maximal tube type subdomains of $\mathcal{X}_{m,n}$ naturally embed into $\calS$. Those boundaries are called \emph{$m$-chains}. A measurable map $\varphi:\partial_\infty \mathbb{H}^p_{\mathbb{C}} \rightarrow \calS$ will be \emph{chain-preserving} if it sends chains to chains. 

If $\omega\in \Omega^2(\mathcal{X}_{m,n})^{\su}$ is the \emph{K\"{a}hler form} of $\mathcal{X}_{m,n}$, denoting by $\mathcal{X}^{(3)}_{m,n}$ the set of distinct triples of points in $\mathcal{X}_{m,n}$, we can define
the following function
\begin{align*}
\begin{array}{cccc}
\beta : & \mathcal{X}_{m,n}^{(3)} & \longrightarrow & \matR \\
 &(x,y,z) & \longmapsto & \frac{1}{\pi}\int_{\Delta(x,y,z)}  \omega
\end{array}
\end{align*}
where $\Delta(x,y,z)$ is any triangle with vertices $x,y,z$ and geodesic edges. 
Since $\omega$ is $\su$-invariant, it is closed by Cartan's Lemma \cite[Lemma VII.4]{Hel01}, and using Stokes' Theorem we get that $\beta$ is a well-defined $\su$-invariant bounded cocycle. Clerc and \O ersted \cite{CO} proved that it is possible to extend $\beta$ to triple of points of the Shilov boundary $\calS$ in a measurable way getting a map $\beta_{\calS}$ called \emph{Bergmann cocycle}. In this way we obtain a bounded $\su$-invariant measurable cocycle, that is $\beta_{\calS} \in \mathcal{B}^\infty((\calS)^3;\matR)^{\su}$, and hence by Section \ref{section_cohomology} we obtain a class in $\Hcb^2(\su;\matR)$. 

\begin{defn}
The bounded $\su$-invariant measurable cocycle $\beta_{\calS}$ is the \emph{Bergmann cocycle} and the class determined by the Bergmann cocycle in $\Hcb^2(\su;\matR)$ is called \emph{K\"{a}hler class}. 
\end{defn}

We recall some properties of the Bergmann cocycle listed in \cite[Proposition 2.1]{Pozzetti} and that we will use later in the proof of the main theorem. The following hold
\begin{itemize}
 \item[\textit{(1)}] The map $\beta_{\calS}$ is an alternating cocycle defined everywhere such that
 $$|\beta_{\calS}(\xi_0,\xi_1,\xi_2)|\leq \textup{rk}(\mathcal{X}_{m,n})$$
 for every $\xi_0,\xi_1,\xi_2$ in $\calS$;
 \item[\textit{(2)}]  The map $\beta_{\calS}$ attains its maximum only on triples of distinct points lying on $m$-chains. Such triples are called \emph{maximal}.
\end{itemize}

In the case $m=p$ and $n=1$ the Bergmann cocycle boils down to the Cartan angular invariant $c_p$ (see \cite{Goldmancomplex}) and we will call the associated cohomology class the \emph{Cartan class}.

\subsection{Ergodic, smooth and amenable actions}\label{section_actions}
In this section we are going to recall the definitions of ergodic, smooth and amenable actions. 
These notions will be crucial in the proof of Theorem \ref{main_theorem}. We refer the reader to \cite{zimmer:libro}, where all those definitions are discussed with more details.

In order to define both smooth and ergodic actions we first need to introduce the notion of countably separated space.

\begin{defn}\label{separable_definition}
 A Borel space $(X,\calB)$ is \emph{countably separated} if there exists a countable family of Borel sets
 $\{B_j\}_{j \in J}$ that separate points.
\end{defn}
A relevant example of countably separated space is the quotient space of an algebraic variety defined over a local field of characteristic zero by an algebraic action of an algebraic group. This is a consequence of  \cite[Theorem 2.1.14]{zimmer:libro} together with \cite[Proposition 3.1.3]{zimmer:libro}. 

Using the notion of countably separated space we are ready to define the concept of smooth action.

\begin{defn}\label{smooth_definition}
 Let $(X,\calB)$ be a countably separated $G$-space. The action is called \emph{smooth} if the quotient Borel structure on $X/G$ is countably
 separated.
\end{defn}

Smooth actions are crucial in the study of boundary theory. Indeed one of the key point of the proofs of both Margulis \cite{margulis:super} and Zimmer \cite{zimmer:annals} superrigidity results relies on the smoothness of the action of product groups on the set of rational functions between boundaries. To be more precise, we are going to give an explicit example in our context. Denote by $G=\pu$ and by $H=\su$. Recall that $G$ can be seen as the real points of its complexification $\textbf{G}=\textup{PSL}(p+1,\mathbb{C})$ once we have suitably fixed
a real structure on it. A similar thing holds for $H$ and its complexification $\mathbf{H}=\textup{SL}(m+n,\mathbb{C})$. Additionally, there exist parabolic subgroups $\mathbf{P} < \mathbf{G}$ and $\mathbf{Q} < \mathbf{H}$ for which $\partial \mathbb{H}^p_{\mathbb{C}}=(\mathbf{G}/\mathbf{P})(\mathbb{R})$ and $\mathcal{S}_{m,n}=(\mathbf{H}/\mathbf{Q})(\mathbb{R})$. We say that a map between $\partial\matH$ and $\calS$ is \emph{rational}, if it is the restriction of a rational map between $\mathbf{G}/\mathbf{P}$ and $\mathbf{H}/\mathbf{Q}$. This enables us to speak about the set $\calQ\coloneqq\textup{Rat}(\partial\matH,\calS)$ of rational maps between $\partial\matH$ and $\calS$. It is possible to define a joint action of $G$ and $H$ as follows
$$((g,h)\cdot f)) (\xi)\coloneqq h\cdot f(g^{-1} \xi) \ ,$$
for each $g \in G$, $h\in H$ and $f\in \calQ$. Following \cite[Proposition 3.3.2]{zimmer:libro} we have that the actions of $G$, $H$ and $G\times H$ on $\calQ$ are all smooth. 

We now move on to the definition of ergodic actions. 
\begin{defn}
 Let $G$ be a locally compact second countable group and let $(X,\mu)$ be a Borel probability $G$-space.
 The action is \emph{ergodic} if for every $G$-invariant Borel set $A$ we have either $\mu(A)=0$ or $\mu(X\setminus A)=0$.
 \end{defn}

Ergodicity can be translated in terms of measurable invariant functions. Indeed an action of $G$ on $X$ is ergodic if and only if for every countably separated space $Y$, every $G$-invariant map $f:X\rightarrow Y$ is essentially constant \cite[Proposition 2.1.11]{zimmer:libro}. 

We conclude this brief section by recalling the notion of amenable spaces.
%
%
%

\begin{defn}\label{amenable_action}
Let $G$ be a locally compact second countable group. Let $(S,\mu)$ be a Lebesgue $G$-space. A \emph{mean} on $\Linf (G\times S;\matR)$ is a $G$-equivariant $\Linf(S;\matR)$-linear operator 
$$m: \Linf (G\times S;\matR)\rightarrow \Linf (S;\matR) \ , $$ 
which has norm one, it is positive and it satisfies $m(\chi_{G \times S})=\chi_S$. An action of $G$ on $S$ is \emph{amenable}, or equivalently $S$ is an \emph{amenable $G$-space}, if there exists a mean on $\Linf (G\times S;\matR)$. 
\end{defn}

Actions determined by amenable groups are amenable, but more generally one can characterize the amenability of a group using actions. Indeed any group acting amenably on a space with finite invariant measure is amenable \cite[Proposition 4.3.3]{zimmer:libro}. The crucial property that we are going to use later is given by the fact that, given a closed subgroup $H \leq G$, then $H$ is amenable if and only if the $G$-action of the quotient space $G/H$ is amenable \cite[Proposition 4.3.2]{zimmer:libro}. In particular, given a semisimple Lie group $G$ of non-compact type, since any minimal parabolic subgroup $P \leq G$ is amenable, the quotient $G/P$ must be an amenable $G$-space.

\subsection{Relative metric ergodicity, boundaries and boundary maps}\label{section_relative_isometric_ergodic}
The goal of this section is to introduce the notion of boundary for a generic locally compact and second countable group $\Gamma$. This definition is due to Bader and Furman and we refer to \cite{BF14} for further details. We point out that in \cite{BF14} the authors first provide the more general definition of boundary pairs and then the one of $\Gamma$-boundary, but here we will not deal with pairs. The objects that we introduce in this section are the main characters of Theorem \ref{thm_boundary_map} and of Section \ref{section_boundary_map}.

We start with the following first refinement of the notion of ergodic space.
\begin{defn}
A Lebesgue $\Gamma$-space $(X,\mu_X)$ is \emph{metrically ergodic} if for any isometric action $\Gamma\rightarrow \Isom(M,d)$ on a separable metric space $(M,d)$, any $\Gamma$-equivariant measurable map $X\rightarrow M$ is essentially constant. 
\end{defn}
We note that a metrically ergodic action is actually ergodic, taking as $M$ the space $\{0,1\}$ with the trivial $\Gamma$-action. For our goal, we need a further refinement of metric ergodicity. Before doing that, we give the definition of relative $\Gamma$-isometric action.
\begin{defn}
A metric on a Borel function $p:M\rightarrow T$ between standard Borel probability spaces is a function $d:M\times_T M \rightarrow [0,\infty)$ whose restriction $d_{|p^{-1}(t)}$ on each $p$-fiber is a separable metric.

Given a metric on $p:M \rightarrow T$, an action of $\Gamma$ on $M$ is \emph{fiber-wise isometric} if there exists a $p$-compatible $\Gamma$-action on $T$ such that, for any $t\in T, \;x,y\in p^{-1}(t)$ and $\gamma\in \Gamma$ we have
$$d_{|p^{-1}(\gamma t)}(\gamma x,\gamma y)=d_{|p^{-1}(t)}(x,y).$$
\end{defn}
The notion of fiberwise isometric action allows us to introduce the following  
\begin{defn}\label{def:rel:iso:ergodic}
A map $q:X\rightarrow Y$ between Lebesgue $\Gamma$-space is \emph{relatively metrically ergodic} if for any fiber-wise isometric $\Gamma$-action on $p:M\rightarrow T$ and measurable $\Gamma$-equivariant maps $f:X\rightarrow M$ and $g:Y\rightarrow T$ there exists a measurable $\Gamma$-equivariant map $\psi:Y\rightarrow M$ such that the following diagram commutes
\begin{center}
\begin{tikzcd}
X \arrow{r}{f}\arrow{d}{q}     & M \arrow{d}{p}   \\
Y \arrow{r}{g}\arrow[dotted]{ru}{\psi}    & T.
\end{tikzcd}
\end{center}
\end{defn}
It is worth noticing that relative metric ergodicity boils down to metric ergodicity if we consider the trivial projection $q:X \rightarrow\{ \ast \}$ on a point.

We finally have all the needed tools to give the notion of boundary. 
\begin{defn}\label{definition_boundary}
A $\Gamma$-\emph{boundary} is a Lebesgue $\Gamma$-space $B$ such that the projections $\textup{pr}_1:B\times B\rightarrow B$ and $\textup{pr}_2:B\times B\rightarrow B$ on the two factors are relative metric ergodic.
\end{defn}

\begin{oss}
As observed in \cite[Remarks 2.4]{BF14} a $\Gamma$-boundary in the sense of Definition \ref{definition_boundary} is a \emph{strong $\Gamma$-boundary} in the sense of Burger and Monod \cite{burger2:articolo}.
\end{oss}

Bader and Furman \cite[Theorem 2.3]{BF14} proved that for any lattice $\Gamma$ in a connected semi-simple Lie group $G$ of non-compact type, the quotient $G/P$ is a $\Gamma$-boundary, where $P$ is a minimal parabolic subgroup. This leads naturally to the following 

\begin{defn}
Let $\Gamma <G$ be a torsion-free lattice in a semi-simple Lie group of non-compact type and let $H$ be a locally compact group. Consider a standard Borel probability $\Gamma$-space $(X,\mu_X)$ and a Lebesgue $H$-space $(Y,\nu)$. A \emph{boundary map} for a measurable cocycle $\sigma:\Gamma \times X \rightarrow H$ is a measurable map 
$$
\phi:G/P \times X \rightarrow Y \ ,
$$
which is \emph{$\sigma$-equivariant}, that is
$$
\phi(\gamma \xi,\gamma x)=\sigma(\gamma,x)\phi(\xi,x) \ ,
$$
for all $\gamma \in \Gamma$ and almost every $\xi \in G/P, x \in X$. 
\end{defn}

\begin{oss}\label{remark_boundary}
We now compare the notion of boundary map with the ones of Section \ref{section_cocycles}.
\begin{itemize}
\item[(i)] If $\rho:G \rightarrow H$ is a representation and $\sigma=\sigma_{\rho}$ is the induced cocycle as in Definition \ref{def_induced_cocycle}, a $\rho$-equivariant map $\varphi:G/P\rightarrow Y$ naturally defines a $\sigma_{\rho}$-equivariant map $\phi:G/P\times X\rightarrow Y$ as 
$$\phi(\xi,x)\coloneqq\phi(\xi)$$
for every $\xi\in G/P$ and $x\in X$.
\item[(ii)] If $\phi:G/P\times X\rightarrow Y$ is a boundary map for a cocycle $\sigma: G\times X\rightarrow H$ and $f:X	\rightarrow G$ is a measurable function,  the map $\phi^f:G/P\times X\rightarrow Y$ defined as 
$$\phi^f(\xi,x)\coloneqq f(x)^{-1}\phi(\xi,x)$$ is a boundary map for the twisted cocycle $\sigma^f$ introduced in Definition \ref{def_cohomologous_cocycle}.
\end{itemize}

\end{oss}

The existence of boundary maps for Zariski dense cocycles will be a crucial result in the proof of our main theorem.

\section{Toledo invariant associated to a cocycle}\label{section_toledo}
In this section we define the Toledo invariant associated to a measurable cocycle,
generalizing the standard definition given for representations. 
Firstly, we do not use boundary maps, as done by the second author in \cite{savini2020} and then we follow the approach adopted by the second author and Moraschini
in \cite{savini3:articolo,moraschini:savini,moraschini:savini:2,savini2020} for the definition of multiplicative constants. Indeed the Toledo invariant will be a particular case of multiplicative constant in the sense of \cite{moraschini:savini:2}. 

We are going to fix the following 
\begin{setup}\label{setup}
 Consider $p \geq 2$ and $1 \leq m \leq n$. We assume that:
	\begin{itemize}
	\item $\Gamma \leq \pu$ is a torsion-free lattice;
	\item $(X,\mu_X)$ is a standard Borel probability $\Gamma$-space;
	\item $\sigma:\Gamma \times X \rightarrow \su$ is a measurable cocycle.
	\end{itemize}
\end{setup}

As proved in \cite[Lemma 2.7]{savini2020}, the map
$$
\textup{C}_{\textup{b}}^\bullet(\sigma):\textup{C}_{\textup{cb}}^\bullet(\su;\mathbb{R})^{\su} \rightarrow \textup{C}_{\textup{b}}^\bullet(\Gamma;\mathbb{R})^{\Gamma} \ ,
$$
$$
\psi\mapsto \textup{C}_{\textup{b}}^\bullet(\sigma)(\gamma_0,\ldots,\gamma_\bullet):=\int_X \psi(\sigma(\gamma_0^{-1},x)^{-1},\ldots,\sigma(\gamma_{\bullet}^{-1},x)^{-1})d\mu_X(x) 
$$
is a well-defined cochain map and hence
induces the following map at a cohomological level
$$
\Hb^\bullet(\sigma): \Hcb^\bullet(\su;\matR) \rightarrow \Hb^\bullet(\Gamma;\matR) \ , 
$$
$$
\Hb^\bullet(\sigma)([\psi]):=[\textup{C}_{\textup{b}}^\bullet(\sigma)(\psi)] \, .
$$
Then we can compose with the \emph{transfer map} associated to $\Gamma$. The latter, denoted by
$$\Tbbullet:\Hb^{\bullet}(\Gamma;\matR)\rightarrow \Hcb^{\bullet}(\pu;\matR) \ ,$$
it is the map induced in cohomology by
$$
\widetilde{\Tbbullet}:\textup{C}_{\textup{b}}^\bullet(\Gamma;\mathbb{R})^{\Gamma}\rightarrow\textup{C}_{\textup{cb}}^\bullet(\pu;\mathbb{R})^{\pu}\; ,
$$
$$
(\widetilde{\Tbbullet}\psi)(\gamma_0,\ldots,\gamma_{\bullet}) \coloneqq\int_{\Gamma\setminus\pu}  \psi (\overline{g}\gamma_0,\ldots,\overline{g}\gamma_{\bullet})d\mu(\overline{g}) \, .
$$
Here $\mu$ is the probability measure induced on the quotient by the Haar measure on $\pu$ (see \cite{bucher2:articolo,moraschini:savini,moraschini:savini:2} for more details about the transfer map).

Since $\Hcb^2(\pu;\matR) \cong \matR$ and it is generated by the Cartan class, we can give the following
\begin{defn}
In the situation of Setup \ref{setup}, the \emph{Toledo invariant associated to} $\sigma$ is the real number $\textup{t}_b({\sigma})$ satisfying 
\begin{equation} \label{toledo_equation}
\Tbdue(\Hb^2(\sigma)(k^b_{\su}))=\textup{t}_b(\sigma) [c_p] \ ,
\end{equation}
where $k^b_{\su}$ is the K\"{a}hler class of $\su$ and $[c_p]$ is the Cartan class.
\end{defn}  

If $\sigma$ admits a boundary map $\phi:G/P\times X\rightarrow \calS$, on can implement the pullback in an alternative way.
Thanks to \cite[Lemma 4.2]{moraschini:savini} we know that the map defined as 
$$
\textup{C}^\bullet(\Phi^X):\mathcal{B}^\infty((\calS)^{\bullet+1};\mathbb{R})^{\su} \rightarrow \textup{L}^\infty((\partial_\infty \mathbb{H}^p_{\mathbb{C}})^{\bullet+1};\mathbb{R})^\Gamma \ ,
$$
$$
\psi \mapsto \textup{C}^\bullet(\Phi^X)(\psi)(\xi_0,\ldots,\xi_\bullet):=\int_X \psi(\phi(\xi_0,x),\ldots,\phi(\xi_\bullet,x))d\mu_X(x) \ ,
$$
 induces the following map at a cohomological level
$$
\textup{H}^\bullet(\Phi^X): \textup{H}^\bullet( \mathcal{B}^\infty((\calS)^{\bullet+1};\matR)^{\su}) \rightarrow \Hb^\bullet(\Gamma;\matR) \ , 
$$
$$
\textup{H}^\bullet(\Phi^X)([\psi]):=[\textup{C}^\bullet(\Phi^X)(\psi)] \ .
$$
In degree two, the composition of the above map with the transfer map $\Tbdue$ applied to the class $[\beta_{\calS}]$ induced by the Bergmann cocycle defines a class $\Tbdue(\Hm^2(\Phi^X))([\beta_{\calS}])$ in $\Hm^2(\textup{L}^\infty((\partial_\infty \mathbb{H}^p_{\mathbb{C}})^{\bullet+1};\mathbb{R})^{\pu})\cong \Hcb^2(\pu;\matR)$.
It follows by \cite[Lemma 2.10]{savini2020} that the following diagram commutes
\begin{equation}
\begin{tikzcd}\label{diagram}
\Hm^{\bullet}(\calB^{\infty} ((\calS)^{\bullet+1};\matR)^{\su})\arrow{r}{\Hm^{\bullet}(\Phi^X)}\arrow{d}{\mathfrak{c}^{\bullet}}
& \Hb^{\bullet}(\Gamma;\matR)\arrow[stealth-]{ld}[d]{\Hb^{\bullet}(\sigma)}
\\
\Hcb^{\bullet}(\su;\matR)
\end{tikzcd}
\end{equation}
and hence $\Tbdue(\Hb^2(\sigma))(k_{\su}^b)=\Tbdue(\Hm^2(\Phi^X))([\beta_{\calS}])$

Even if the Toledo invariant is defined independently of a boundary map, we prefer to maintain here its expression in terms of boundary maps since this formulation is crucial in the proof of our superrigidity result.
For our purpose, we want to rewrite Equation \eqref{toledo_equation} that can be at the level of cochains in terms of $\phi$ and $\beta_{\calS}$. Since the transfer map is also induced by the map
$$
\widehat{\Tbbullet}:\Linf((\partial\matH)^{\bullet+1});\matR)^{\Gamma}\rightarrow\Linf((\partial\matH)^{\bullet+1});\matR)^{\pu}\,
$$
$$
(\widehat{\Tbbullet}\psi)(\xi_0,\ldots,\xi_{\bullet}) \coloneqq\int_{\Gamma\setminus\pu}  \psi (\overline{g}\xi,\ldots,\overline{g}\xi_{\bullet})d\mu(\overline{g}) \, ,
$$
the diagram \eqref{diagram} leads to the following formula
\begin{align}\label{formula}
 \int\limits_{\Gamma\backslash \pu}\int\limits_X & \beta_{\calS}(\phi(\overline{g}\xi_0,x),\phi(\overline{g}\xi_1,x),\phi(\overline{g}\xi_2,x))d\mu_X(x)d\mu(\overline{g})=\\
&= \textup{t}_b(\sigma) c_p(\xi_0,\xi_1,\xi_2) \nonumber.
\end{align}
Moreover, as proved for instance in \cite{bucher2:articolo,Pozzetti,BBIborel}, Equation \eqref{formula} holds for every triple $(\xi_0,\xi_1,\xi_2)$ of pairwise distinct points in $(\partial \matH)^3$.

\begin{oss}
It is worth noticing that Equation \eqref{formula} is a suitable adaptation of \cite[Proposition 1]{moraschini:savini:2} to this particular context. The absence of coboundary terms is due to the doubly ergodic action of $\Gamma$ on the boundary $\partial\matH$ and to the fact that all the considered cochains are alternating. Additionally, the Toledo invariant $\textup{t}_b(\sigma)$ is the multiplicative constant $\lambda_{\beta_{\calS},c_p}(\sigma)$ associated to $\sigma, \beta_{\calS}, c_p$, namely
$$
\textup{t}_b(\sigma)=\lambda_{\beta_{\calS},c_p}(\sigma) \ ,
$$
according to \cite[Definition 3.16]{moraschini:savini:2}.
\end{oss}

\begin{prop}\label{toledo_properties}
In the situation of Setup \ref{setup}, the Toledo invariant $\textup{t}_b(\sigma)$ satisfies:
 \begin{enumerate}
  \item[\textit{(1)}] $|\textup{t}_b(\sigma)|\leq \rk(\mathcal{X}_{m,n})$;
  \item[\textit{(2)}] $|\textup{t}_b(\sigma)|=\rk(\mathcal{X}_{m,n})$ if and only if the slice $\phi_x\coloneqq \phi(\ \cdot \ ,x)$ is chain-preserving for almost every $ x\in X$.
 \end{enumerate}
\end{prop}

 \begin{proof}
\textit{Ad 1.}  By Section \ref{section_bergmann_cartan} we know that $||c_p||_{\infty}\leq 1$ and that $||\beta_{\calS}||_{\infty}\leq \rk(\mathcal{X}_{m,n})$.
Hence we obtain
\begin{align*}
|\textup{t}_b(\sigma)|&=\lVert \textup{t}_b(\sigma) c_p\rVert_{\infty}=\lVert \widehat{\Tbdue}(\textup{C}^2(\Phi^X)(\beta_{\calS})) \rVert_{\infty}\leq \rk(\mathcal{X}_{m,n}) \ ,
\end{align*}
since both the transfer map $\widehat{\textup{T}^2_b}$ and the pullback map $\textup{C}^2(\Phi^X)$ are norm non-increasing.

\textit{Ad 2.} Assume that the slice $\phi_x$ is chain preserving for almost every $x \in X$. Fixed a point $x\in X$, if $\phi_x$ is chain preserving and the triple $(\xi_0,\xi_1,\xi_2)$ lies on a chain, then the triple 
$(\phi_x(\overline{g}\xi_0),\phi_x(\overline{g}\xi_1),\phi_x(\overline{g}\xi_2))$
lies on a $m$-chain
for almost every $\overline{g} \in \Gamma \backslash \pu$. Hence, if we fix a triple $(\xi_0,\xi_1,\xi_2) \in (\partial\matH)^{(3)}$ of positive points on a chain, it holds $c_p(\xi_0,\xi_1,\xi_2)=1$ and by hypothesis it follows 
$$
\beta_{\calS}(\phi_x(\overline{g}\xi_0),\phi_x(\overline{g}\xi_1),\phi_x(\overline{g}\xi_2))=\rk(\mathcal{X}_{m,n})
$$
for almost every $\overline{g} \in \Gamma \backslash \pu, x \in X$. In this way we obtain
\begin{align*}
\textup{t}_b(\sigma)= &\int_{\Gamma\backslash\pu}\left( \int_X \beta(\phi_x(\overline{g}\xi_0),\phi_x(\overline{g}\xi_1),\phi_x(\overline{g}\xi_2))d\mu_X(x)\right)d\mu(\overline{g})= \\
 = & \int_{\Gamma\backslash\pu}\left( \int_X \rk(\mathcal{X}_{m,n}) d\mu_X(x)\right)d\mu(\overline{g})= \rk(\mathcal{X}_{m,n}) \ ,
\end{align*}
as claimed. 

For the converse assume $\textup{t}_b(\sigma)=\rk(\mathcal{X}_{m,n})$. Fixing a positive triple $(\xi_0,\xi_1,\xi_2) \in (\partial\matH)^{(3)}$ on a chain, it follows by Equation \eqref{formula} that, 
$$\beta(\phi_x(\overline{g}\xi_0),\phi_x(\overline{g}\xi_1),\phi_x(\overline{g}\xi_2))=\rk(\mathcal{X}_{m,n})$$
for almost every $\overline{g}\in \Gamma\setminus\pu$ and $x\in X$.
By the $\sigma$-equivariance of $\phi$ we argue that 
$$\beta(\phi_x(g\xi_0),\phi_x(g\xi_1),\phi_x(g\xi_2))=\rk(\mathcal{X}_{m,n}) \ ,$$
for almost every $g \in \pu$ and $x \in X$. By the transitivity of the $\pu$-action on chains, the map $\phi_x$ is chain preserving, as desired.  

The same arguments can be used for the negative case.
\end{proof}

By Proposition \ref{toledo_properties} it follows naturally the next
\begin{defn}\label{def:maximal:cocycles}
In the situation of Setup \ref{setup}, a cocycle $\sigma: \Gamma\times X\rightarrow \su$ is \emph{maximal} if $\textup{t}_b(\sigma)=\rk(\mathcal{X}_{m,n})$.
\end{defn}

It is worth mentioning that the notion of maximal measurable cocycles is a substantial extension of that one of maximal representations. 
Indeed, given any maximal $\rho:\Gamma \rightarrow \su$ in the sense of Pozzetti \cite{Pozzetti} and any
measurable function $f: X \rightarrow \su$, it is easy to check that the twisted cocycle $\sigma_{\rho}^f$ is actually maximal. 
Moreover, if $\rho$ is Zariski dense then it admits an essentially Zariski dense boundary map $\varphi:\partial\matH\rightarrow \calS$ \cite[Proposition 2.9]{Pozzetti}. Hence the induced boundary map $\phi:\partial\matH\times X\rightarrow \calS$ defined as in Remark \ref{remark_boundary} has in fact essentially Zariski dense slices. In particular it satisfies the hypothesis of Theorem \ref{main_theorem}. Hence our main theorem can be seen as the converse of 
what noticed above. 

We conclude this section with a characterization of boundary maps associated to maximal cocycles.
\begin{lemma}\label{lemma_rational_map}
In the situation of Setup \ref{setup}, if $\sigma$ is maximal and the slice $\phi_x$ has essentially Zariski dense image for almost every $x\in X$, then $\phi_x$ is rational for almost every $x \in X$.
\end{lemma} 
\begin{proof}
It follows by \cite[Theorem 1.6]{Pozzetti} since $\phi_x$ is essentially Zariski dense for almost $x\in X$ and it is chain preserving by 
Proposition \ref{toledo_properties}. 
\end{proof}

\section{Boundary maps}\label{section_boundary_map}
In this section our aim is to investigate the existence of a boundary map for cocycle with specific properties. In particular we exploit \cite[Theorem 5.3]{BFalgebraic} to prove Theorem \ref{thm_boundary_map}. From this we will argue the existence of boundary maps in our specific context. Then we will show that Zariski density of a measurable cocycle and the ergodicity of the standard Borel space imply that the boundary map must have Zariski dense slices. 

We recall that, by \cite[Theorem 5.3]{BFalgebraic},  in the setting of Theorem \ref{thm_boundary_map}
for any ergodic Lebesgue $\Gamma$-space $Y$ there exists an algebraic subgroup $L<H$ and a $\Gamma$-equivariant universal map 
$\phi:Y\rightarrow H/L$ such that, for any algebraic $H$-space $V$ and for any $\Gamma$-equivariant measurable map $\psi:Y \rightarrow V$, there exists a $\Gamma$-equivariant measurable map $\pi:L/H\rightarrow V$ which makes the following diagram commutative
\begin{center}
\begin{tikzcd}
Y\arrow{rr}{\phi}\arrow{rd}{\psi}&&H/L\arrow{ld}[u]{\pi}\\
& V.
\end{tikzcd}
\end{center}
This universal property is the fundamental ingredient in the proof of Theorem \ref{thm_boundary_map}.

\begin{proof}[Proof of Theorem \ref{thm_boundary_map}]
Since $B$ is a strong boundary \cite[Remarks 2.4]{BF14}, by \cite[Proposition 2.4]{MonShal0} both $B\times X$ and $B\times B\times X$ are ergodic $\Gamma$-space, thus we denote by $L,L_0$ the algebraic subgroups of $H$ and by $\phi:B\times X\rightarrow H/L, \ \phi_0:B\times B\times X\rightarrow H/L_0$ the $\Gamma$-equivariant universal maps associated respectively to $B\times X$ and to $B\times B\times X$.

Since $B\times X$ is amenable \cite[Proposition 4.3.4]{zimmer:libro}, then there exists a $\sigma$-equivariant map 
$\nu:B\times X\rightarrow \Prob(H/P)$ where $P<H$ is a minimal parabolic subgroup and $\Prob(H/P)$ is the space of probability measures on $H/P$.

By ergodicity of $\Gamma$ on $B\times X$ and by the smooth action of $H$ on $\Prob(H/P)$ \cite[Corollary 3.2.23]{zimmer:libro}, it follows that the induced map 
$$\bar{\nu}:B\times X\rightarrow \Prob(H/P)/H$$ is essentially constant. 
Equivalently, $\nu$ has image essentially contained in a single $H$-orbit, namely we get a map
$B\times X\rightarrow H/\textup{Stab}_H(\mu_0)$ where $\textup{Stab}_H(\mu_0)$ is the stabilizer in $H$ of some probability measure $\mu_0\in\Prob(H/P)$. By \cite[Corollary 3.2.23]{zimmer:libro} we have that $\textup{Stab}_H(\mu_0)$ is algebraic and amenable. Hence we can exploit the universal property of $\phi$, in order to get a $\Gamma$-equivariant map $H/L\rightarrow H/\textup{Stab}_H(\mu_0)$. Thus, up to conjugacy, $L<\textup{Stab}_H(\mu_0)$ and moreover, by amenability of $\textup{Stab}_H(\mu_0)$, it follows that $L$ is amenable. 

Consider now the map $\phi\circ\pi_2$ where $\pi_2:B\times B\times X\rightarrow B\times X$ is the projection on the last two factors. By the universal property of $\phi_0$, we get a $\Gamma$-equivariant map $\pi:H/L_0\rightarrow H/L$ such that the following diagram commutes
\begin{center}
\begin{tikzcd}
B\times B\times X\arrow{rr}{\phi_0}\arrow{rd}{\phi\circ\pi_2}&&H/L_0\arrow{ld}[u]{\pi}\\
& H/L.
\end{tikzcd}
\end{center}

Again, up to conjugation, we can assume that $L_0<L$ and, denoting by $$R:=\textup{Rad}_u(L)$$ the unipotent radical of $L$, we get the chain of inclusion $L_0<L_0R<L$ and the induced chain of projections \begin{tikzcd}
H/L_0\arrow{r}{p_1}&H/L_0R\arrow{r}{p_2}&H/L.
\end{tikzcd}

Define now the maps
$$\Phi:B\rightarrow \Meas(X,H/L),\;\;\;\;\Phi(\xi)(\cdot)\coloneqq\phi(\xi,\cdot)$$ and 
$$\Phi_0:B\times B\rightarrow \Meas(X,H/L_0),\;\;\;\;\Phi_0(\xi_1,\xi_2)(\cdot)\coloneqq \phi_0(\xi_1,\xi_2,\cdot).$$
Hence, for every $\gamma\in \Gamma$ and for almost every $\xi\in B$ we have
\begin{align*}
\Phi(\gamma\xi)(\cdot)&=\phi(\gamma\xi,\cdot)\\
&=\phi(\gamma\xi,\gamma\gamma^{-1}\cdot)\\
&=\sigma(\gamma,\gamma^{-1}\cdot)\phi(\xi,\gamma^{-1}\cdot)\\
&=\sigma(\gamma,\gamma^{-1}\cdot)\Phi(\xi)(\gamma^{-1}\cdot).
\end{align*}
Thus, if we define an action of $\Gamma$ on $\Meas(X,H/L)$ as 
\begin{equation}\label{gamma_action}
(\gamma f)(\cdot)\coloneqq\sigma(\gamma,\gamma^{-1}\cdot) f(\gamma^{-1}\cdot)
\end{equation}
for any $\gamma\in \Gamma$ and $f \in \Meas(X,H/L)$, 
we actually showed that $\Phi$ is a $\Gamma$-equivariant map. Similarly one can check the $\Gamma$-equivariance of $\Phi_0$.

Consider now the following commutative diagram
\begin{center}
\begin{tikzcd}
B\times B\arrow{r}[u]{\Phi_0}\arrow{d}[d]{\textup{pr}_2} & \Meas(X,H/L_0R) \arrow{d}[]{p_2^X}\\
B\arrow{r}[u]{\Phi}\arrow[dotted]{ur}{\Psi}&\Meas(X,H/L)
\end{tikzcd}
 \end{center}
 where $p_2^X(f)(\cdot)\coloneqq p_2(f(\cdot))$, and $\textup{pr}_2:B\times B\rightarrow B$ is the projection on the second factor.
 We remark that the existence of the map $\Psi$ follows from
 the fact that $p_2^X$ is fiberwise $\Gamma$-isometric
 and from the relative metric ergodicity of $\textup{pr}_2$.
 In fact, a metric on $\Meas(X,H/L_0R)$ compatible with the $\Gamma$-action defined in \eqref{gamma_action}, can be found as follows. Let $\mathbf{e} \in \Meas(X,H/L)$ be the constant function $\mathbf{e}(x):=L$. If we denote by $d$ the $L/R$-invariant metric on $L/L_0R$ cited in the proof of \cite[Theorem 3.4]{BF14}, we can set the metric $d_0$ on the special fiber $(p^X_2)^{-1}(\mathbf{e}) \cong \Meas(X,L/L_0R)$ as  
 \begin{equation*}
 d_0(f,g)\coloneqq \int_X \frac{d(f(x),g(x))}{1+d(f(x),g(x))} d\mu_X(x)
 \end{equation*}
for every $f,g\in\Meas(X,L/L_0R)$. 
Since the group $\Meas(X,H)$ acts transitively on $\Meas(X,H/L_0R)$, we can move the metric $d_0$ on the whole $\Meas(X,H/L_0R)$. To show the compatibility of the collection of metrics on the fibers, let $h \in \Meas(X,H/L)$ and let $f,g \in (p^X_2)^{-1}(h)$. Then 
\begin{align*}
d_{\gamma. h}(\gamma f, \gamma g)&=\int_X \frac{d_{\gamma.h}(\sigma(\gamma,\gamma^{-1}x)f(\gamma^{-1}x),\sigma(\gamma,\gamma^{-1}x)g(\gamma^{-1}x))}{1+d_{\gamma.h}(\sigma(\gamma,\gamma^{-1}x)f(\gamma^{-1}x),\sigma(\gamma,\gamma^{-1}x)g(\gamma^{-1}x))}d\mu_X(x) =\\
&=\int_X \frac{d_{h}(f(\gamma^{-1}x),g(\gamma^{-1}x))}{1+d_{h}(f(\gamma^{-1}x),g(\gamma^{-1}x))}d\mu_X(x)=\\
&=\int_X \frac{d_{h}(f(x),g(x))}{1+d_{h}(f(x),g(x))}d\mu_X(x)=d_h(f,g) \ ,
\end{align*}
where we used the transitivity of $\Meas(X,H)$ and the definition of the metrics on the fibers to move from the first line to the second one and we concluded exploiting the $\Gamma$-invariance of $\mu_X$. 

Define the $\Gamma$-equivariant map $\psi:B\times X\rightarrow H/L_0R$ as $\psi(\xi,x)\coloneqq \Psi(\xi)(x)$ for almost every $\xi\in B$ and almost every $x\in X$. By the universal property of $\phi$, there exists $q:H/L\rightarrow  H/L_0R$ which is in fact a isomorphism, and hence, up to conjugation, we can assume that $L_0R=L$. 

By defining 
$$
\phi \time \phi: B \times B \times X \rightarrow H/L \times H/L, \ \ (\phi \times \phi)(\xi_1,\xi_2,x):=(\phi(\xi_1,x),\phi(\xi_2,x)) \ ,
$$
we know by the universal property of $\phi_0$ that we have the following commutative diagram
\begin{center}
\begin{tikzcd}
B\times B\times X\arrow{rr}{\phi_0}\arrow{rd}{\phi \times \phi}&&H/L_0\arrow{ld}[u]{ }\\
& H/L \times H/L.
\end{tikzcd}
\end{center}

Additionally, notice that given $\gamma_1,\gamma_2 \in \Gamma$ we have 
$$
(\phi \times \phi)(\gamma_1 \xi_1,\gamma_2 \xi_2,x)=(\sigma(\gamma_1,\gamma_1^{-1}x)\phi(\xi_1,\gamma^{-1}x),\sigma(\gamma_2,\gamma_2^{-1}x)\phi(\xi_2,\gamma_2^{-1}x)) \ ,
$$
for almost every $\xi_1,\xi_2 \in B, x \in X$ by the $\sigma$-equivariance of $\phi$. 
As a consequence of the Zariski density of $\sigma$, the essential image of $\phi \times \phi$ is Zariski dense in $H/L \times H/L$. Thus $H/L_0$ is Zariski dense in $H/L \times H/L$ or, equivalently, $RL_0R$ is Zariski dense in $H$. Thus, by \cite[Lemma 3.5]{BF14} $H$ is parabolic and, being amenable, is also minimal. This concludes the proof. 

\end{proof}
 

\begin{oss}
Given a Zariski dense measurable cocycle $\sigma:\Gamma \times X \rightarrow G$, where $\Gamma < \textup{PU}(1,1)$ is a lattice and $G$ is a Hermitian Lie group, one the author \cite[Section 4.1]{savini2020} conjectured that $\sigma$ has a boundary map. Theorem \ref{thm_boundary_map} gives an answer to that question and has important consequences on the study of measurable cocycles of surface groups. Indeed, in the case that $\sigma$ is also maximal in the sense of Definition \ref{def:maximal:cocycles}, then by \cite[Theorem 1]{savini2020} the target must be a group of \emph{tube type}. In this way one obtains an true generalization of \cite[Theorem 5]{BIW1} to the context of measurable cocycles. 
\end{oss}

We can now prove the existence of a boundary map in our specific setting.

\begin{cor}\label{cor_boundary_map}
Let $p \geq 2$ and let $\Gamma < \pu$ a torsion-free lattice. Consider an ergodic standard Borel probability $\Gamma$-space. If $\sigma:\Gamma \times X \rightarrow \su$ is a Zariski dense measurable cocycle, then there exists a boundary map $\phi:\partial_\infty \mathbb{H}^p_{\mathbb{C}} \times X \rightarrow \calS$. 
\end{cor}

\begin{proof}
By \cite[Theorem 2.3]{BF14} the Furstenberg boundary of $\Gamma$, which coincides with the visual boundary $\partial\matH$, is actually a $\Gamma$-boundary.
Hence, applying Theorem \ref{thm_boundary_map} we get a map into the Furstenberg boundary of $\su$ and we compose with the projection on the Shilov boundary (induced by the inclusion of a minimal parabolic subgroup into a maximal one).
\end{proof}

\begin{oss}
Thanks to Corollary \ref{cor_boundary_map}, in Setup \ref{setup} one can in fact drop the condition of existence of a boundary map for Zariski dense cocycles. Throughout the paper, when Setup \ref{setup} will be recalled, we will tacitly consider the boundary map provided by Corollary \ref{cor_boundary_map}. 
\end{oss}

Since the essential Zariski density of the slices of a boundary map will be needed in the proof of the main theorem, we are going to prove that property in the next 

\begin{prop}\label{prop:zariski:density}
In the setting of Setup \ref{setup} we assume that $\Gamma$ is ergodic on $X$ and that $\sigma$ is Zariski dense. Then for almost every $x \in X$ the slice $\phi_x$ is essentially Zariski dense. 
\end{prop}

\begin{proof}
Before starting the proof, recall that the Shilov boundary $\mathcal{S}_{m,n}$ corresponds to the real points $\mathcal{S}_{m,n}=(\mathbf{H}/\mathbf{Q})(\mathbb{R})$ of the quotient of the complexification $\mathbf{H}$ of $\textup{SU}(m,n)$ (which is $\textup{SL}(m+n,\mathbb{C})$) modulo a maximal parabolic subgroup $\mathbf{Q}$ stabilizing a maximal isotropic subspace of $\mathbb{C}^{m+n}$. For almost every $x \in X$, we are going to denote by $\mathbf{V}_x \subset \mathbf{H}/\mathbf{Q}$ the smallest Zariski closed set such that $V_x:=\mathbf{V}_x(\mathbb{R}) \subset \mathcal{S}_{m,n}$ and $\phi_x^{-1}(V_x)$ has full measure in $\partial \mathbb{H}^n_{\mathbb{C}}$. As noticed in \cite{Pozzetti} those sets exist by the Noetherian property.


By embedding suitably $\mathbf{H}/\mathbf{Q}$ in some complex projective space $\mathbb{P}^{N}(\mathbb{C})$, we can define a map, 
$$
\mathfrak{v}: X \rightarrow \textup{Var}(\mathbb{P}^N(\mathbb{C})) \ , \ \ \ \mathfrak{v}(x):=\mathbf{V}_x \ .
$$
Here $ \textup{Var}(\mathbb{P}^N(\mathbb{C}))$ is the set of all the possible closed varieties inside $\mathbb{P}^N(\mathbb{C})$ with the measurable structure coming from the Hausdorff metric (Zariski closed sets are closed in the Euclidean topology and this makes sense). The map $\mathfrak{v}$ is measurable since the slice $\phi_x$ varies measurably with respect to $x \in X$ as a consequence of \cite[Chapter VII, Lemma 1.3]{margulis:libro}, by the measurability of $\phi$. By the same proof of \cite[Proposition 4.2]{sarti:savini}, the map $\mathfrak{v}$ is also $\sigma$-equivariant, that is 
$$
\mathfrak{v}(\gamma x)=\sigma(\gamma,x)\mathfrak{v}(x) \ ,
$$
or equivalently 
\begin{equation}\label{eq:v:equivariant}
\textbf{V}_{\gamma x}=\sigma(\gamma,x)\textbf{V}_x \ ,
\end{equation}
for every $\gamma \in \Gamma$ and almost every $x \in X$. 

On $\textup{Var}(\mathbb{P}^N(\mathbb{C}))$ the group $\textup{GL}(N+1,\mathbb{C})$ acts naturally on the left. As noticed in the proof of \cite[Proposition 3.3.2]{zimmer:libro}, the set $\textup{Var}(\mathbb{P}^N(\mathbb{C}))$ decomposes as a countable union of varieties and the action of $\textup{GL}(N+1,\mathbb{C})$ on those varieties is algebraic and hence smooth. Seeing $\su$ as a subgroup of $\textup{GL}(N+1,\mathbb{C})$, we argue that the quotient $\Sigma:=\textup{Var}((\mathbb{P}^N(\mathbb{C}))/\textup{SU}(m,n)$ is countably separated and $\mathfrak{v}$ induces a map 
$$
\overline{\mathfrak{v}}: X \rightarrow \Sigma \ , \ \ \overline{\mathfrak{v}}(x):=\su \cdot \mathbf{V}_x \ , 
$$
which is $\Gamma$-invariant, since $\mathfrak{v}$ was $\sigma$-equivariant by Equation \eqref{eq:v:equivariant}. By the ergodicity of $\Gamma$ on $X$, the above map must be essentially constant. Equivalently $\mathfrak{v}$ must take values into a unique orbit $\textup{SU}(m,n) \cdot \mathbf{V}_{x_0}$, for some $x_0 \in X$. Hence there must exist a measurable map $g:X \rightarrow \textup{SU}(m,n)$ such that 
$$
\mathbf{V}_x=g(x)\mathbf{V}_{x_0} \ ,
$$
and the measurability of $g$ follows by the measurability of $\phi$ and by \cite[Corollary A.8]{zimmer:libro}. This implies that $\sigma$ is cohomologous to a cocycle preserving $\mathbf{V}_{x_0}$ and the latter must coincide with $\mathbf{H}/\mathbf{Q}$ by the Zariski density assumption on $\sigma$. Hence for almost every $x \in X$, we have $\mathbf{V}_x=\mathbf{H}/\mathbf{Q}$ and hence $V_x=\mathcal{S}_{m,n}$, which means that almost every slice is essentially Zariski dense. 
\end{proof}

\section{Proof of the main theorem}\label{section_proof}
The aim of this section is to prove Theorem \ref{main_theorem}. The proof follows the line of that in \cite[Theorem 4.1]{zimmer:annals} 
and is based on both
Lemma \ref{lemma_rational_map} and the following useful result about invariant measures on quotients of algebraic groups. Before stating the lemma, recall that a $\mathbb{R}$-group is an algebraic group whose defining equations are given by polynomials with real coefficients. 

\begin{lemma}\label{lemma_measure}
 Let $\mathbf{G}$ be a semisimple algebraic $\mathbb{R}$-group and let $\mathbf{G}_0$ be a $\mathbb{R}$-subgroup. Denote by $G=\mathbf{G}(\mathbb{R})$ and $G_0=\mathbf{G}_0(\mathbb{R})$ the associated real points, respectively. Consider a lattice $\Gamma $ in $G$.
 Then, any measure on $G/G_0$ which is invariant by left translations in $\Gamma$, it is also a $G$-invariant measure.
\end{lemma}
\begin{proof}
 Since $\mathbf{G}$ is an affine algebraic group, by Chevalley's Theorem \cite[Proposition 3.1.3]{zimmer:libro} there exists a suitable positive integer $N$ and a rational representation $\pi: \mathbf{G} \rightarrow \textup{PGL}(N,\mathbb{C})$ defined over $\mathbb{R}$ such that the image $\pi(G_0)$ coincides with the stabilizer in $G$ of a line $\ell \subset \mathbb{R}^{N}$.
 Thus we get a map
 $$\overline{\pi}:G/G_0 \rightarrow G\cdot [\ell] \subseteq \mathbb{P}^{N-1}(\mathbb{R}), \;\;\;\; g\cdot G_0 \mapsto \pi(g)[\ell] \ .$$ 

%


Consider now a measure $\mu$ on $G/G_0$. Its push-forward measure $\nu:=\overline{\pi}_*\mu$ on $\mathbb{P}^{N-1}(\mathbb{R})$ is supported on the orbit $G\cdot [\ell]$.
By Zimmer \cite[Theorem 3.2.4]{zimmer:libro} the stabilizer $$L:=\Stab _{\textup{PGL}(N,\mathbb{C})}(\nu)$$ corresponds to the real points of an algebraic group $\mathbf{L} < \textup{PGL}(N,\mathbb{C})$ defined over $\mathbb{R}$. Since the representation $\pi$ is rational, the preimage $$\mathbf{H}=\pi^{-1}(\textbf{L})$$ is a Zariski closed subgroup of $\mathbf{G}$. The intersection $$H:=\textbf{H} \cap G$$ coincides with the stabilizer $\Stab_G(\mu)$. By hypothesis we have that $\Gamma \subseteq \Stab_G(\mu)=\mathbf{H} \cap G$. Passing to the Zariski closures, we get
$$\mathbf{G}=\overline{\Gamma}^Z\subseteq \overline{\Stab_G(\mu)}^Z \subseteq \mathbf{H} \ ,$$
where the first equality follows by the Borel Density Theorem \cite[Theorem 3.2.5]{zimmer:libro}. 
Hence the stabilizer of $\mu$ in $G$ is the whole group and we are done.
\end{proof}

We are now able to give the proof of the main theorem.

\begin{proof}[Proof of Theorem \ref{main_theorem}]

Assuming the same algebraic structures on $\partial\matH$ and $\calS$ as those ones described in Section \ref{section_actions}, 
we denote the set of rational maps between boundaries by
$$\calQ\coloneqq \textup{Rat} (\partial \matH,\calS) \  .$$
As described in Section \ref{section_actions}, there exists a natural action of $\pu\times \su$ on it defined as follows:
for each $h \in \pu$, $g\in \su, \xi\in \partial\matH$ and $f\in \calQ$,
$$((h,g)\cdot f)) (\xi)\coloneqq g\cdot f(h^{-1} \xi) \ .$$

Since $\sigma$ is Zariski dense, by Corollary \ref{cor_boundary_map} we know that there exists a boundary map $\phi:\partial_\infty \mathbb{H}^p_{\mathbb{C}} \times X \rightarrow \calS$. 
Being $\sigma$ also maximal, by Lemma \ref{lemma_rational_map} we can define the function
$$\Phi:X\rightarrow \calQ, \;\;\;\; x\mapsto \phi_x$$
and by composing it with the projection $\calQ\rightarrow\calQ/\su$ we obtain 
$$\widehat{\Phi}:X\rightarrow \widehat{\calQ}\coloneqq\calQ/\su, \;\;\;\; x\mapsto \su\cdot\phi_x \ .$$
Since $\phi$ is a boundary map for $\sigma$, its equivariance implies 
\begin{align}\label{equivariance}
 \Phi(\gamma x)&=\phi_{\gamma x}(\ \cdot \ ) =\\
\nonumber &= \phi(\cdot,\gamma x)= \\
\nonumber &= \phi(\gamma\gamma^{-1}\cdot , \gamma x) = \\
\nonumber &= \sigma(\gamma,x)\phi (\gamma^{-1}\cdot , x) = \\
\nonumber &= \sigma(\gamma,x)(\gamma\Phi(x)).
\end{align}
In the equation above, notice that $\gamma \in \Gamma$ acts on the quotient $\widehat{\mathcal{Q}}$ via $\gamma \cdot (\su \cdot \psi):=\su \cdot (\gamma \cdot \psi)$, where $\gamma \cdot \psi$ is the rational map $(\gamma \cdot \psi) (\xi)=\psi(\gamma^{-1} \xi)$, for $\xi \in \partial \mathbb{H}^{p}_{\mathbb{C}}$. As a consequence of Equation \eqref{equivariance} we get
$$
\Phi(\gamma x) \in \su \cdot \gamma \cdot \Phi(x) \ ,
$$
and hence it holds 
$$
\widehat{\Phi}(\gamma x)= \gamma \cdot \widehat{\Phi}(x) \ ,
$$
from which we deduce that $\widehat{\Phi}$ is a $\Gamma$-equivariant map on the quotient.
It follows that the induced map
$$\widehat{\widehat{\Phi}}:X\rightarrow \widehat{\widehat{\mathcal{Q}}}:=\calQ/\pu\times\su, \;\;\;\; x\mapsto \pu\times\su\cdot\phi_x.$$
is $\Gamma$-invariant and, since $\Gamma$ acts ergodically on $X$,
it is essentially constant or, equivalently, $\widehat{\Phi}$ takes values in a single $\pu$-orbit. Notice that to conclude that $\widehat{\widehat{\Phi}}$ is essentially constant, we exploited the fact that $\widehat{\widehat{\mathcal{Q}}}$ is countably separated because of the smoothness of the joint action of both $\pu$ and $\su$ on $\mathcal{Q}$ (see Section \ref{section_actions}). 

Let $x_0\in X$ be a point such that $\widehat{\Phi}$ takes value in the orbit $\pu \cdot \widehat{\Phi}(x_0)$ and set $G_0\coloneqq \Stab_{\pu}(\widehat{\Phi}(x_0))$. The latter is an algebraic $\mathbb{R}$-subgroup of $\pu$ by \cite[Proposition 3.3.2]{zimmer:libro}. Since the orbit $\pu \cdot \widehat{\Phi}(x_0)$ may be identified with $\pu/G_0$ by the smoothness of the action \cite[Theorem 2.1.14]{zimmer:libro}, we can compose the map 
$$
\widehat{\Phi}: X \rightarrow \pu \cdot \widehat{\Phi}(x_0) \cong \pu/G_0 \ ,
$$
with a measurable section
$$
s:\pu/G_0 \rightarrow \pu \ ,
$$
which exists by \cite[Corollary A.8]{zimmer:libro}. The previous composition gives us a map 
\begin{equation*}
g: X\rightarrow \pu
\end{equation*}
which is measurable (being the composition of measurable maps) and such that 
$$\widehat{\Phi}(x)=g(x)\widehat{\Phi}(x_0)$$
for almost every $x\in X$.
By definition, we have
\begin{align*}
\widehat{\Phi}(\gamma x)=g(\gamma x)\widehat{\Phi}(x_0)
\end{align*}
for every $\gamma\in \Gamma$ and almost every $x\in X$. On the other hand, by equivariance we get
\begin{align*}
 \widehat{\Phi}(\gamma x)=\gamma(\widehat{\Phi}(x))
\end{align*}
and thus 
$$(\gamma g(x))^{-1}g(\gamma x) \in G_0.$$
The induced map 
\begin{equation*}
\bar{g}: X\rightarrow \pu/G_0
\end{equation*}
is $\Gamma$-equivariant and this ensures the existence of a $\Gamma$-invariant finite measure (by push-forward) on $\pu/G_0$.
By Lemma \ref{lemma_measure}, such a measure is in fact $\pu$-invariant and, since
$G_0$ is a closed subgroup, it coincides with $\pu$ again by the Borel Density Theorem \cite[Theorem 3.2.5]{zimmer:libro}.
Hence $\widehat{\Phi}$ is essentially constant or, equivalently, 
$\Phi$ takes values in a single $\su$-orbit.
Denote again by $\phi_0$ an element in the orbit and choose a map
$$f:X \rightarrow \su$$
satisfying
$$\Phi(x)=f (x) \phi_0.$$
The measurability of $f$ follows by the same argument we used to prove the measurability of $g$. By rewriting Equation \eqref{equivariance} using $f$ we obtain 
\begin{equation}
f(\gamma x)\phi_0=\sigma(\gamma,x) f(x)\gamma\phi_0 
\end{equation}
and then
\begin{equation}\label{formula4}
\gamma^{-1}\phi_0= f(\gamma x)^{-1}\sigma(\gamma,x) f(x)\phi_0.
\end{equation}
We define 
$$\beta:\Gamma\times X\rightarrow \su, \;\;\;\; \beta(\gamma,x)\coloneqq f(\gamma x)^{-1}\sigma(\gamma,x) f(x)$$
and, by Equation \eqref{formula4}, we get
$$\phi_0(\xi)=\beta(\gamma,x_1)^{-1}\beta(\gamma,x_2)\phi_0(\xi)$$
for all $\gamma\in\Gamma$ and for almost all $\xi\in\partial\matH, x_1,x_2\in X$. Hence 
$\beta(\gamma,x_1)^{-1}\beta(\gamma,x_2)$ lies in the stabilizer of the image of $\phi_0$. Since the latter is essentially Zariski dense, the product $\beta(\gamma,x_1)^{-1}\beta(\gamma,x_2)$ actually stabilizes $\calS$. Since the pointwise stabilizer of $\calS$ (that is the kernel of the action of $\su$ on $\mathcal{S}_{m,n}$) is trivial, it follows that $\beta$ does not depend on $X$ and hence it is the cocycle associated to a representation 
$$\rho:\Gamma\rightarrow \su \ .$$
Moreover, by Equation \eqref{formula4}, the map $\phi_0$ is a boundary map for $\rho$, it is rational and has  
essentially Zariski dense image in $\su$. It follows by \cite[Theorem 1.1]{Pozzetti} that $\rho$ is the 
restriction of a representation 
$$\widetilde{\rho}:\pu\rightarrow \su$$
and, finally, $\sigma$ is cohomologous to the restriction to $\Gamma$ of the induced cocycle $\sigma_{\tilde \rho}$, as desired.
\end{proof}

We can now prove that, except when either $m=1$ or $m=n$, there are no maximal Zariski dense cocycle as in the statement of Theorem \ref{main_theorem}.
\begin{proof}[Proof of Proposition \ref{corollary}.]
Following the proof of Theorem \ref{main_theorem}, given such a maximal cocycle, there exists a maximal representation
$\rho:\Gamma\rightarrow \su$. Following \cite[Corollary 1.2]{Pozzetti}, if $m\neq n$, such a representation cannot exist.
\end{proof}

Since maximal measurable cocycles into $\su$ cannot be Zariski dense when $1<m<n$, it is quite natural to ask which could be their algebraic hull. We answered to this question in  \cite{sarti:savini} and we report here both the statement and the proof for sake of completeness. One can see that the characterization we obtained is completely analogous to the one given in \cite[Theorem 1.3]{Pozzetti}. 

\begin{prop} \cite[Proposition 1.2]{sarti:savini} \label{prop:algebraic:hull}
In the setting of Setup \ref{setup} we assume that $1<m<n$, that $X$ is $\Gamma$-ergodic and that $\sigma$ is maximal.  Denoting by $\textup{\textbf{H}}$ the algebraic hull of $\sigma$ and by $H=\textup{\textbf{H}}(\matR)^{\circ}$, then $H$ splits as the product $K\times L_{nt}\times L_t$, where $K$ is a compact subgroup of $\su$,  $L_t$ is a Hermitian Lie group of tube type and $L_{nt}$ is a Hermitian Lie group not of tube type that splits again as a product of several copies of $\textup{SU}(n,1)$.
\end{prop}

\begin{proof}
Recall that a measurable cocycle $\sigma:\Gamma \times X \rightarrow \su$ is \emph{tight} if it satisfies
$$
\lVert \textup{H}^2(\Phi^X)(\beta_{\calS})\rVert_\infty= m \ .
$$
Since we assumed $\sigma$ maximal, it must be tight by the same proof of \cite[Proposition 4.7]{sarti:savini}. By \cite[Theorem 3.5]{savini2020} the group $H$ is reductive and hence it splits as the product of a compact factor $L_c=K$ 
 and a non compact factor $L_{nc}$. 
By splitting $L_{nc}$ in simple factors $L_1,\ldots,L_k$, we notice that the composition of 
 $\sigma$ with any projection $\pi_i: L_1\times\ldots\times L_k\rightarrow L_i$ is a Zariski dense maximal measurable cocycle from 
 a complex hyperbolic lattice to $L_i$.
 It follows by \cite[Theorem 5]{moraschini:savini:2} that none of the $L_i$'s can be isomorphic to $\textup{SU}(1,1)$. 
 Hence the inclusion $L_{nc}\rightarrow \textup{SU}(m,n)$ satisfies the hypothesis of \cite[Proposition 2.5]{Pozzetti}, which states that each factor 
 $L_i$ is either of tube type or isomorphic to some $\textup{SU}(p_i,q_i)$.
 We denote by $L_t$ the tube-type part and we focus on the non-tube type factors.
 Again by \cite[Theorem 5]{moraschini:savini:2}, if one of $\textup{SU}(p_i,q_i)$'s is actually of the form $\textup{SU}(s,1)$ (that is $q_i=1$), we must have that $s$ is equal to $n$ by Zariski density. By Corollary \ref{corollary} the Zariski density of an ergodic cocycle taking values into $\textup{SU}(p_i,q_i)$ implies necessarily that $q_i=1$ and we are done. 
\end{proof}

%
One can investigate the behavior of measurable cocycles of complex hyperbolic lattices into other Hermitian Lie groups than $\su$. In \cite{sarti:savini} we study the case of isometries of infinite dimensional Hermitian symmetric spaces. A hint for working with infinite dimensional measurable cocycles was given by the recent paper by Duchesne, L\'{e}cureux and Pozzetti \cite{duchesne:pozzetti} that investigate maximal representations of surface groups and hyperbolic lattices into infinite dimensional Hermitian groups. 

\bibliographystyle{amsalpha}

\bibliography{biblionote}

\end{document}